\documentclass[11pt]{amsart}
\usepackage{mathrsfs,latexsym,amsfonts,amssymb,curves,epic}
\newtheorem{theorem}{Theorem}[section]
\newtheorem{lemma}[theorem]{Lemma}
\newtheorem{corollary}[theorem]{Corollary}
\newtheorem{question}[theorem]{Question}

\theoremstyle{definition}
\newtheorem{definition}[theorem]{Definition}
\newtheorem{proposition}[theorem]{Proposition}
\newtheorem{example}[theorem]{Example}

\begin{document}

\title[The countable type properties in free paratopological groups]
{The countable type properties in free paratopological groups}

\author{Fucai Lin}
\address{(Fucai Lin): School of mathematics and statistics,
Minnan Normal University, Zhangzhou 363000, P. R. China}
\email{linfucai2008@aliyun.com}

\author{Chuan Liu}
\address{(Chuan Liu): Department of Mathematics,
Ohio University Zanesville Campus, Zanesville, OH 43701, USA}
\email{liuc1@ohio.edu}

\author{Kexiu Zhang}
\address{(Kexiu Zhang) School of mathematics and statistics,
Minnan Normal University, Zhangzhou 363000, P. R. China}
\email{kxzhang2006@163.com}

\thanks{The first author is supported by the NSFC (Nos. 11571158, 11471153),
  the Natural Science Foundation of Fujian Province (Nos. 2016J05014, 2016J01671, 2016J01672) of China.}

\keywords{free paratopological group; countable type; almost countable type; almost subcountable type.}
\subjclass[2000]{primary 22A30; secondary 54D10; 54E99; 54H99}

\begin{abstract}
A space $X$ is of {\it countable type} (resp. {\it subcountable type}) if every
compact subspace $F$ of $X$ is contained in a compact subspace
$K$ that is of countable character (resp. countable pseudocharacter) in $X$. In this paper, we mainly show that: (1) For a functionally Hausdorff space $X$, the free paratopological group $FP(X)$ and the free abelian paratopological group $AP(X)$ are of countable type if and only if $X$ is discrete; (2) For a functionally Hausdorff space $X$, if the free abelian paratopological group $AP(X)$ is of subcountable type then $X$ has countable pseudocharacter. Moreover, we also show that, for an arbitrary Hausdorff $\mu$-space $X$, if $AP_{2}(X)$ or $FP_{2}(X)$ is locally compact, then $X$ is homeomorphic to the topological sum of a compact space and a discrete space.
\end{abstract}

\maketitle

\section{Introduction}
{\bf All spaces are $T_0$ unless stated otherwise.} We denote by $\mathbb{N}$ the set of all natural
numbers, $\mathbb{Z}$ the set of integer and $\omega$ the first countable order. The letter $e$
denotes the neutral element of a group. For a space $X$, we always denote the set of all isolated points of $X$ by $\mbox{I(X)}$. Readers may consult
\cite{AT2008,E1989,Gr1984} for notations and terminology not
explicitly given here.

The subsequent work in this paper is at least in part provoked by the desire to understand
more fully the inter-relationship between the algebraic and topological conditions
imposed in such celebrated theorems about topological groups as the following:

\begin{theorem}{\bf Pontryagin Theorem} \cite{PO1973}
Each $T_0$ topological group is Tychonoff.
\end{theorem}

\begin{theorem}{\bf Birkhoff-Kakutani Theorem} \cite{BG1936,KS1936}
Each first-countable topological group is metrizable
\end{theorem}

\begin{theorem}\cite{RW1981}
A topological group $G$ is a paracompact $p$-space if and only if there exists a compact subset that is of countable character in $G$.
\end{theorem}

\begin{theorem}\cite{Ar1981}
A topological group $G$ is of countable pseudocharacter if and only if it is submetrizable.
\end{theorem}

A {\it paratopological group} $G$ is a group $G$ with a topology such that
the product map of $G \times G$ into $G$ is continuous. If $G$ is a paratopological group  and the inverse operation of $G$ is continuous, then $G$ is called a {\it topological group}.  However, there exists a paratopological group which is not a
topological group; Sorgenfrey line (\cite[Example
1.2.2]{E1989}) is such an example.  Paratopological groups were discussed and many results have been obtained in \cite{AR2005,AT2008,EN2012,EN2013,Lin,LFC2012,Lin2012,LC,LC1,LL2010,PN2006,RS}.

\begin{definition}\cite{RS}
Let $X$ be a subspace of a paratopological group $G$. Assume that
\begin{enumerate}
\item The set $X$ generates $G$ algebraically, that is $\langle X\rangle=G$;

\item  Each continuous mapping $f: X\rightarrow H$ to a paratopological group $H$ extends to a continuous homomorphism $\hat{f}: G\rightarrow H$.
\end{enumerate}
Then $G$ is called the {\it Markov free paratopological group on} $X$ and is denoted by $FP(X)$.
\end{definition}

Again, if all the groups in the above definitions are Abelian, then we get the definition of the {\it Markov free Abelian paratopological group on} $X$ which is denoted by $AP(X)$.

Finally, we recall and introduce some concepts.

\begin{definition}
Let $X$ be a space. Then
\begin{enumerate}
\item $X$ is of {\it countable type} \cite{E1989} if every
compact subspace $F$ of $X$ is contained in a compact subspace
$K$ that is of countable character in $X$;

\item $X$ is of {\it pointwise countable type} \cite{E1989} if every
point of $X$ is contained in a compact subspace
$K$ that is of countable character in $X$;

\item $X$ is of {\it subcountable type} \cite{VA1970} if every
compact subspace $F$ of $X$ is contained in a compact $G_{\delta}$-subspace $K$ of $X$;

\item $X$ is of {\it pointwise subcountable type} \cite{VA1970} if every
point of $X$ is contained in a compact $G_{\delta}$-subspace $K$ of $X$;

\item $X$ is called {\it almost countable type} if it contains a non-empty compact set $K$ that is of countable character in $X$;

\item $X$ is called {\it almost subcountable type} if it contains a non-empty compact set $K$ that is a $G_{\delta}$-set in $X$.
\end{enumerate}
\end{definition}

It is easy to see that

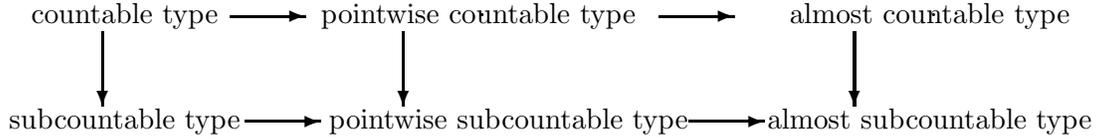
\setlength{\unitlength}{1cm}
\begin{picture}(15,2)\thicklines
 \put(12,0){\makebox(0,0){almost subcountable type}}
 \put(2.7,1.4){\vector(1,0){1}}
  \put(8.4,1.4){\vector(1,0){1}}
  \put(8.8,0){\vector(1,0){1}}
 \put(1.3,0){\makebox(0,0){subcountable type}}
 \put(2.9,0){\vector(1,0){1}}
 \put(1,1.2){\vector(0,-1){1}}
 \put(5,1.2){\vector(0,-1){1}}
 \put(11,1.2){\vector(0,-1){1}}
 \put(1.3,1.4){\makebox(0,0){countable type}}
 \put(6,1.4){\makebox(0,0){pointwise countable type}.}
 \put(6.4,0){\makebox(0,0){pointwise subcountable type}}
 \put(12,1.4){\makebox(0,0){almost countable type}.}
\end{picture}

\bigskip

{\bf Note} (1) In general, a space is of subcountable type need not be of countable type. For example, every space with a countable network is of subcountable type, but it need not be of countable type (take a countable dense subspace of $D^{\omega_{1}}$).

(2) There exists a pointwise countable type space $X$ which is not of countable type, see \cite[Example 2.9]{VA1970}.

(3) There exists a pointwise subcountable type space $X$ which is not of pointwise countable type, see Example~\ref{e0}.

(4) There exists an almost subcountable type space $X$ which is not of almost countable type, see Example~\ref{e1}.

(5) There exists an almost subcountable type space $X$ which is not of pointwise subcountable type, see Example~\ref{e2}.

(6) There exists a pointwise subcountable type space $X$ which is not of subcountable type, see \cite[Example 7.4]{VA1970}.

(7) In \cite{VA1970}, pointwise subcountable type and subcountable type  are named pseudopoint-countable type and pseudocountable type, respectively.

\section{Some properties of almost countable type spaces or almost subcountable type spaces}
The following proposition is easy to check.

\begin{proposition}
The product $X=\prod_{n\in\mathbb{N}}X_{n}$ of countably many almost countable type spaces (resp. almost subcountable type spaces)
 is of almost countable type (resp. almost subcountable type).
\end{proposition}

Let $X, Y$ be two spaces. The mapping $f: X\rightarrow Y$ is {\it perfect} if $f$ is a closed continuous mapping and $f^{-1}(y)$ is compact for each $y\in Y$.

\begin{proposition}\label{p6}
If $f: X\rightarrow Y$ is a perfect mapping from a space $X$ onto an almost countable type space (resp. an almost subcountable type space) $Y$, then $X$ is of almost countable type (resp. almost subcountable type).
\end{proposition}

\begin{proof}
Since $Y$ is of almost countable type, there exist a compact subset $K$ in $Y$ and a sequence of open neighborhoods $\{U_{n}: n\in\mathbb{N}\}$ that is a neighborhood base for $Y$ at $K$. Then $f^{-1}(K)$ is compact in $X$ since $f$ is a perfect mapping \cite[Theorem 3.7.2]{E1989}. We claim that $\{f^{-1}(U_{n}): n\in\mathbb{N}\}$ is a neighborhood base for $X$ at $f^{-1}(K)$. Indeed, let $U$ be an arbitrary open neighborhood of $f^{-1}(K)$ in $X$. For each $y\in K$, since $f$ is a closed mapping, there exists an open neighborhood $W_{y}$ of $y$ in $Y$ such that $f^{-1}(W_{y})\subset U$. Then the family $\{W_{y}: y\in K\}$ covers $K$, and hence there exists $n\in \mathbb{N}$ such that $K\subset U_{n}\subset\bigcup\{W_{y}: y\in K\}$. Therefore, we have $$f^{-1}(K)\subset f^{-1}(U_{n})\subset f^{-1}(\bigcup \{W_{y}: y\in K\})=\bigcup\{f^{-1}(W_{y}): y\in K\})\subset U.$$

Obviously, if $Y$ is of almost subcountable type, then so $X$ is .
\end{proof}

\begin{example}\label{e0}
There exists a pointwise subcountable type space $G$ which  is not of pointwise countable type.
\end{example}

\begin{proof}
Let $X$ be a non-discrete and non-compact metrizable space. Then the free topological group $G=F(X)$ is submetrizable \cite{AT2008}, hence it is of pointwise subcountable type. We claim that $G$ is not of pointwise countable type. Suppose not, it follows from \cite{T2003} that $X$ is compact or discrete, which is a contradiction.
\end{proof}

\begin{example}\label{e1}
There exists an almost subcountable type paratopological group $G$ which is not of almost countable type.
\end{example}

\begin{proof}
Let $X$ be a Tychonoff submetrizable non-discrete space. Then the free paratopological group $G=FP(X)$ over $X$ is submetrizable \cite{PN2006}, and hence it is of almost subcountable type. It follows from Theorem~\ref{t2} that $FP(X)$ is not of almost countable type.
\end{proof}

A space is a {\it functionally Hausdorff space} if for two arbitrary distinct points $x$ and $y$ there is a continuous real-valued mapping $f$ on $X$ such that $f(x)\neq f(y)$.

\begin{example}\label{e2}
There exists an almost subcountable type space $G$ which is not of pointwise subcountable type.
\end{example}

\begin{proof}
Let $X$ be a functionally Hausdorff space with uncountable pseudocharacter. It follows from Theorem~\ref{t3} that $AP(X)$ is not of almost subcountable type. Let $G$ be the topological sum of $AP(X)$ and a discrete space $Y$, where $Y\cap AP(X)=\emptyset$. Then $G$ is of almost subcountable type. However, it is obvious that $G$ is not of pointwise subcountable type.
\end{proof}

\section{Some properties of almost countable type paratopological groups}
Firstly, we show that the properties of countable type, pointwise countable type and almost countable type are consistent in the class of paratopological groups.

In a topological
group, the property of almost countable type is equivalent to the well known condition of almost metrizability. Therefore, it is interesting to consider the property of almost countable type in paratopological groups. Indeed, we have the following proposition.

\begin{proposition}\label{p1}
Let $G$ be a paratopological group. Then
$G$ is of almost countable type if and only if it is of countable type.
\end{proposition}

\begin{proof}
Obviously, it suffices to show that each almost countable type paratopological group is of countable type. Since $G$ is of almost countable type, there exist a compact subset $K$ and a sequence of open subsets $\{U_{n}: n\in\mathbb{N}\}$ such that $\{U_{n}: n\in\mathbb{N}\}$ is a base for $G$ at $K$. Without loss of generality, we may assume that $e\in K.$ Let $L$ be an arbitrary compact subset of $G$. Put $F=L K$.  Clearly, $F$ is compact since $G$ is a paratopological group. Then the family $\{L U_{n}: n\in\mathbb{N}\}$ is a base for $G$ at $F$. Indeed, take an arbitrary open neighborhood $O$ of $F$ in $G$. By \cite[Proposition 1.4.29]{AT2008}, there exists an open neighborhood $V$ of the neutral element $e$ in $G$ such that $LKV =FV\subset O$. Then $KV$ is an open neighborhood of $K$, and therefore there exists $n\in \mathbb{N}$ such that $K\subset U_{n}\subset KV$, then $LK\subset LU_{n}\subset LKV\subset O$. Therefore, $G$ is of countable type.
\end{proof}

\begin{proposition}\label{p0}
Suppose that $H$ is a closed subgroup of an almost countable type paratopological group $G$. Then the quotient space $G/H$ is of almost countable type.
\end{proposition}

\begin{proof}
Let $\pi: G\rightarrow G/H$ be the quotient mapping. Obviously, $\pi$ is open. Since $G$ is of almost countable type, it contains a non-empty compact set $K$ which has a neighborhood base $\{U_{n}: n\in \mathbb{N}\}$ in $G$. Then $\{\pi(U_{n}): n\in \mathbb{N}\}$ is a neighborhood base for $G/H$ at the compact set $\pi(K)$, so $G/H$ is of almost countable type.
\end{proof}

\begin{proposition}\label{p3}
Let $G$ be a Hausdorff paratopological group that is of almost countable type, and let $O$ be a neighborhood of the neutral element in $G$. Then there exists a compact subgroup $H\subset O$ such that $H$ is of countable character in $G$.
\end{proposition}

\begin{proof}
Since paratopological group $G$ is of almost countable type, there exists a non-empty compact set $K$ that is of countable character in $G$. By the homogeneity of $G$, we may assume that $K$ contains the neutral element $e$ of $G$. Let $\{U_{n}: n\in\mathbb{N}\}$ be a neighborhood base for $G$ at $K$. By induction, we can define a sequence $\{V_{n}: n\in \mathbb{N}\}$ of open neighborhoods of $e$ in $G$ satisfying the following conditions:

(i) $V_{1}\subset O$;

(ii)$V_{n+1}^{2}\subset U_{n}\cap V_{n}$ for each $n\in\mathbb{N}$;

(iii) $\overline{V_{n+1}\cap K}\subset K\cap V_{n}$ for each $n\in\mathbb{N}$ (This is possible since $K$ is Hausdorff and compact.).

Put $H=\bigcap_{n\in \mathbb{N}}V_{n}$. It is obvious that $H\subset O.$ Finally, we shall show that $H$ is a compact subgroup in $G$ such that it is of countable character in $G$.

Since $V_{n+1}^{2}\subset U_{n}\cap V_{n}$ for each $n\in\mathbb{N}$, $H$ is a semigroup and $H\subset K$. Note that $$H=H\cap K=(\bigcap_{n\in\mathbb{N}}V_{n})\cap K=\bigcap_{n\in\mathbb{N}}(V_{n}\cap K)=\bigcap_{n\in\mathbb{N}}(\overline{V_{n}\cap K}).$$Thus $H$ is closed, so $H$ is compact.

A non-empty subset $M$ of $H$ is called a right ideal in $H$ if $MH\subset M$. Since $H$ is a compact semigroup of $G$, it contains a minimal closed right ideal, say $L$ (apply the Kuratowski-Zorn lemma to the family of all closed right ideals in $K$ by ordered by inverse inclusion). For each $x\in L$, it is easy to see that $xL\subset LH\subset L$, and thus $xLH\subset x(LH)\subset xL$, i.e., $xL$ is a right ideal in $H$. Since $xL$ is closed in $H$, $xL\subset L$ and $L$ is a minimal closed right ideal in $H$, we have $xL=L$ for each $x\in L$. In particular, we have $x^{2}L=L$ for each $x\in L$, whence it follows that $x^{2}y=x$ for some $y\in L$ and thus $x^{-1}=y\in L$. Therefore, we conclude that $e\in L$, $H=L$, and that $H$ is a subgroup of $G$.

Then it is obvious that compact subgroup $H$ has countable pseudocharacter in $G$ and in $K$. Since $K$ is compact, we conclude that $\chi(H, K)\leq\omega$. Then it follows from \cite[3.1.E]{E1989} that $\chi(H, G)\leq\chi(H, K)\cdot \chi(K, G)$. Therefore, compact subgroup $H$ is of countable character in $G$.
\end{proof}

Similarly, we can show the following proposition.

\begin{proposition}\label{p5}
Let $G$ be a Hausdorff paratopological group that is of almost subcountable type, and let $O$ be a neighborhood of the neutral element in $G$. Then there exists a compact subgroup $H\subset O$ that is of countable pseudocharacter in $G$.
\end{proposition}

\begin{proposition}
Let $G$ be a Hausdorff paratopological group that is of almost countable type, and let $\mathscr{H}$ be the family of all compact subgroups of $G$ which are of countable character in $G$. For each $H\in\mathscr{H}$, let $\pi_{H}: G\rightarrow G/H$ be the quotient mapping onto the left coset space $G/H$. Let $$\mathscr{B}=\{\pi_{H}^{-1}(V): H\in\mathscr{H}\ \mbox{and}\ V\ \mbox{is open in}\ G/H\}.$$ Then $\mathscr{B}$ is a base of $G$.
\end{proposition}

\begin{proof}
Obviously, the family $\mathscr{B}$ is closed under taking left and right translates in $G$. Hence it suffices to show that the subfamily $$\mathscr{B}(e)=\{U\in\mathscr{B}: e\in U\}$$ is a base for $G$ at the neutral element $e$. Let $W$ be a neighborhood of $e$ in $G$. Find a neighborhood $O$ of $e$ such that $O^{2}\subset W$. It follows from Proposition~\ref{p3} that there exists a compact subgroup $H$ of $G$ such that $H\in\mathscr{H}$ and $H\subset O$. Let $V=\pi_{H}(O)$. Since $\pi_{H}^{-1}(V)=\pi_{H}^{-1}\pi_{H}(O)=OH$, the set $V$ is open in $G/H$. Moreover, we have $OH\in\mathscr{B}(e)$ and $OH\subset O^{2}\subset W$.
\end{proof}

\begin{theorem}\label{t0}
Let $G$ be a paratopological group. If $H$ is a compact subgroup of $G$ and $P$ is a closed subset of $G$, then the sets $HP$ and $PH$ are closed in $G$.
\end{theorem}

\begin{proof}
We will prove that $HP$ is closed in $G$. The case of $PH$ differs only in trivial details. Take an arbitrary point $a\not\in HP$. Obviously, $H^{-1}a$  is compact since $H$ is a group, and the sets $H^{-1}a$ and $P$ are disjoint. Then it follows from \cite[Theorem 1.4.29]{AT2008} that there exists an open neighborhood $U$ of $e$ such that $H^{-1}aU$ and $P$ are disjoint. Then it is easy to see that $aU\cap HP=\emptyset$. Since $aU$ is an open neighborhood of $a$, the point $a$ is not in the closure of $HP$. Hence $HP$ is closed.
\end{proof}

By Proposition~\ref{p3} and Theorem~\ref{t0}, one can easily obtain the following two theorems.

\begin{theorem}\label{t1}
Let $G$ be a paratopological group. If $H$ is a compact subgroup in $G$, then the quotient mapping $\pi$ of $G$ onto the quotient space $G/H$ is perfect.
\end{theorem}

\begin{theorem}
Let $G$ be a Hausdorff paratopological group that is of almost countable type. Then there exists a compact subgroup $H$ that is of countable character in $G$ such that the quotient mapping $\pi$ of $G$ onto the quotient space $G/H$ is perfect.
\end{theorem}

\begin{lemma}\label{l2}
Let $G$ be a paratopological group. If $H$ is a compact subgroup that is of countable character in $G$, then $G/H$ is first-countable.
\end{lemma}

\begin{proof}
By the homogeneity of $G$, it suffices to show that the neutral element of $G/H$ is of countable character in $G/H$. Let $\{U_{n}: n\in\mathbb{N}\}$ be a neighborhood base for $G$ at $H$. Since the quotient mapping $\pi: G\rightarrow G/H$ is open, each $\pi(U_{n})$ is an open neighborhood of  the neutral element of $G/H$. Let $W$ be an open neighborhood of  the neutral element of $G/H$. Then $\pi^{-1}(W)$ is an open neighborhood at $H$ in $G$, hence there exists  $n\in\mathbb{N}$ such that $U_{n}\subset \pi^{-1}(W)$, so $\pi(U_{n})\subset W$. Therefore, the family $\{\pi(U_{n}): n\in\mathbb{N}\}$ is a neighborhood base for $G/H$ at the neutral element. Hence $G/H$ is first-countable.
\end{proof}

\begin{theorem}
A Hausdorff paratopological group $G$ is of almost countable type if and only if it contains a compact subgroup $H$ of $G$ such that the left quotient space $G/H$ is first-countable.
\end{theorem}

\begin{proof}
Sufficiency. Suppose that $G$ contains a compact subgroup $H$ of $G$ such that the left quotient space $G/H$ is first-countable. By Theorem~\ref{t1}, it follows that the mapping $\pi: G\rightarrow G/H$ is perfect. Next we show that $H$ is of countable character in $G$. Indeed, let $\{U_{n}: n\in\mathbb{N}\}$ be a neighborhood base for $G/H$ at $\pi(e)$, where $e$ is the neutral element of $G$. For each $n\in\mathbb{N}$, let $V_{n}=\pi^{-1}(U_{n})$. Since the mapping $\pi: G\rightarrow G/H$ is perfect, it is easy to verify that the family $\{V_{n}: n\in\mathbb{N}\}$ is a base for $G$ at $H$. Therefore, $G$ is of almost countable type.

Necessity. Let $G$ be of almost countable type. Then it contains a non-empty compact subgroup $K$ that is of countable character in $G$ by Proposition~\ref{p3}. By Lemma~\ref{l2}, the left quotient space $G/H$ is first-countable.
\end{proof}

\begin{example}
There exists a separable almost countable type paratopological group that
 is not paracompact.
\end{example}

\begin{proof}
Example~4.1 in \cite{Lin2012} is a separable first-countable paratopological group, hence it is of almost countable type. However, it is not paracompact.
\end{proof}

\begin{example}
There exists a Lindel\"{o}f paratopological group
 that is not of almost countable type.
\end{example}

\begin{proof}
Let $X$ be a countable non-discrete $T_{1}$ space. Then the free paratopological group $FP(X)$ is a countable $T_{1}$ space \cite{EN2012}. Obviously, $FP(X)$ is Lindel\"{o}f and of almost subcountable type. However, it is not of almost countable type.
\end{proof}

\section{Almost countable type free paratopological groups}
Now we can prove our main Theorems.

Since $X$ generates the free group $FP_{a}(X)$ \cite{AT2008}, each element $g\in FP_{a}(X)$ has the form $g=x_{1}^{\varepsilon_{1}}\cdots x_{n}^{\varepsilon_{n}}$, where $x_{1}, \cdots, x_{n}\in X$ and $\varepsilon_{1}, \cdots, \varepsilon_{n}=\pm 1$. This word for $g$ is called {\it reduced} if it contains no pair of consecutive symbols of the form $xx^{-1}$ or $x^{-1}x$. It follows that if the word $g$ is reduced and non-empty then it is different from the neutral element of $FP_{a}(X)$. In particular, each element $g\in FP_{a}(X)$ distinct from the neutral element can be uniquely written in the form $g=x_{1}^{\varepsilon_{1}}x_{2}^{\varepsilon_{2}}\cdots x_{n}^{\varepsilon_{n}}$, where $n\geq 1$, $\varepsilon_{i}\in \mathbb{Z}\setminus\{0\}$, $x_{i}\in X$, and $x_{i}\neq x_{i+1}$ for each $i=1, \cdots, n-1$. Similar assertions are valid for $AP_{a}(X)$. For every non-negative integer $n$, denote by $FP_{n}(X)$ and $AP_{n}(X)$ the subspace of paratopological group $FP(X)$ and $AP(X)$ that consists of all words of reduced length $\leq n$ with respect to the free basis $X$, respectively. Define $i_n: \tilde{X}^n\to FP_n(X)$ by $i_n((x_1, x_2, ... x_n))=x_1x_2...x_n$ for each $n$, where $\tilde{X}=X\cup\{e\}\cup X^{-1}$. Then each $i_n$ is a continuous map.

\begin{lemma}\cite{LFC2012}\label{l0}
Let $X$ be a Functionally Hausdorff space, and let $A$ be an arbitrary subset of $AP(X)$. If $A\cap AP_{n}(X)$ is finite for each $n\in\mathbb{N}$, then $A$ is closed and discrete in $AP(X)$.
\end{lemma}

By the group reflexion $G^{\flat}=(G, \tau^{\flat})$ of a paratopological group $(G, \tau)$ we understand the group $G$ endowed with the strongest topology $\tau^{\flat}\subset \tau$ such that $(G, \tau^{\flat})$ is a topological group.

\begin{lemma}\label{l1}\cite{PN2006}
If $X$ is a Functionally Hausdorff space then the topological groups $AP(X)^{\flat}$ and $A(X)$ are topologically isomorphic.
\end{lemma}

\begin{theorem}\label{t2}
The following conditions are equivalent for a functionally Hausdorff space $X$:
\begin{enumerate}
\item The $FP(X)$ is of almost countable type;

\item The $AP(X)$ is of almost countable type;

\item The space $X$ is discrete.
\end{enumerate}
\end{theorem}

\begin{proof}
Since $AP(X)$ is a quotient of $FP(X)$, it follows from Proposition~\ref{p0} that (1)$\Rightarrow$(2). It is also clear that $AP(X)$ and $FP(X)$ over a  discrete space $X$ are discrete. Therefore, it suffices to verify that (2)$\Rightarrow$(3).

By Proposition~\ref{p3}, there exists a compact subgroup $L\subset AP(X)$ that is of countable character in $AP(X)$.

{\bf Claim:} $L=\{0\}$.

Suppose that $L\neq 0$. Take an element $g\in L\setminus\{0\}$. Then $\langle g\rangle$ is an infinite cyclic group of $L$ generated by $g$. Then the intersection $\langle g\rangle \cap AP_{n}(X)$ is finite for each $n\in\omega$, hence it follows from Lemma~\ref{l0} that $\langle g\rangle$ is closed and discrete in $L$. However, $L$  is compact, which is a contradiction.

By Claim, the neutral element 0 is of countable character in $AP(X)$. By \cite{Lin2012}, the neutral element 0 is of countable character in $AP(X)^{\flat}$. It follows from Lemma~\ref{l1} that the free abelian topological group $A(X)$ is first-countable, so $X$ is discrete by Graev's theorem in \cite{GM}.
\end{proof}

By Proposition~\ref{p5}, we can show the following theorem by a similar proof in Theorem~\ref{t2}.

\begin{theorem}\label{t3}
Let $X$ be functionally Hausdorff space. If $AP(X)$ is of almost subcountable type then $X$ is of countable pseudocharacter.
\end{theorem}

\begin{question}
Let $X$ be functionally Hausdorff space that is of countable pseudocharacter. Is $FP(X)$ or $AP(X)$ of almost subcountable type?
\end{question}

\begin{question}
Let $X$ be functionally Hausdorff space. If $FP(X)$ is of almost subcountable type, is $AP(X)$ of almost subcountable type?
\end{question}

\begin{lemma}\label{l3}\cite{EN2013}
If $X$ is a $T_{1}$-space, then the mapping $$i_{2}\mid _{i_{2}^{-1}(FP_{2}(X)\setminus FP_{1}(X))}: i_{2}^{-1}(FP_{2}(X)\setminus FP_{1}(X))\longrightarrow FP_{2}(X)\setminus FP_{1}(X)$$ is a homeomorphism.
\end{lemma}

Recall that a subset $A$ in a space $Z$ is called {\it bounded} (in $Z$) if every real-valued
continuous function on $Z$ is bounded on $A$.

The proof of the following theorem is a modification of \cite[Lemma 2.1]{T2003}. However, for the completeness, we give out the proof.

\begin{theorem}\label{t7}
If $FP_{2}(X)$ is of pointwise countable type, then $X\setminus I(X)$ is bounded in $X$.
\end{theorem}

\begin{proof}
From the hypothesis, there exist in $FP_{2}(X)$ a compact set $K$ containing the neutral element $e$ and a decreasing family of open sets $\{U_{n}: n\in\mathbb{N}\}$ which forms a base for $FP_{2}(X)$ at $K$. Suppose that $X\setminus I(X)$ is not bounded in $X$. Then there exists in $X$ a discrete family of open sets $\{V_{n}: n\in\mathbb{N}\}$ each of which intersects $X\setminus I(X)$. For each $n\in\mathbb{N}$, choose a point $x_{n}\in V_{n}\cap (X\setminus I(X))$. Since, for each $n\in \mathbb{N}$, the point $x_{n}$ is a non-isolated point, it is easy to see that there exists a point $y_{n}\in V_{n}$  with $x_{n}\neq y_{n}$ and $x_{n}^{-1}y_{n}\in U_{n}$. Clearly, we also have $x_{n}^{-1}y_{n}\neq x_{m}^{-1}y_{m}$ for distinct $m, n\in\mathbb{N}$. Put $B=\{x_{n}^{-1}y_{n}: n\in \mathbb{N}\}$.

{\bf Claim:} The set $B$ is closed and discrete in $FP_{2}(X)$.

Indeed, it suffices to show that $B$ is closed in $FP_{2}(X)$. Let $w$ belong to the closure of $B$ in $FP_{2}(X)$. Next we shall show that $w\in B$.

{\bf Case 1:} $w\in X\cup X^{-1}$.

Obviously, the set $X\cup X^{-1}$ is open in $FP_{2}(X)$ and $(X\cup X^{-1})\cap B=\emptyset$, then $w$ is not in the closure of $B$ in $FP_{2}(X)$, this is a contradiction.

{\bf Case 2:} $w=e$.

Since the family $\{V_{n}: n\in\mathbb{N}\}$ is discrete and $x_{n}$ and $y_{n}$ are distinct points in $V_{n}$, we can find a continuous quasi-pseudometric $d$
on $X$ such that $d(x_{n}, y_{n})\geq 1$ for each $n\in\mathbb{N}$. Then the set $O_{2}(d)=\{g\in FP(X): \tilde{d}(e, g)<1\}\cap FP_{2}(X)$ is an open neighborhood of $e$ in $FP_{2}(X)$ disjoint from $B$, where $\tilde{d}$ is the Graev extension of $d$ to $FP(X)$. Thus $w$ is not in the closure of $B$ in $FP_{2}(X)$, this is a contradiction.

{\bf Case 3:} $w=FP_{2}(X)\setminus FP_{1}(X)$.

Clearly, the set $\{(x_{n}^{-1}, y_{n}): n\in \mathbb{N}\}$ is closed in $i_{2}^{-1}(FP_{2}(X)\setminus FP_{1}(X))$. By Lemma~\ref{l3}, it is easy to see that $w\in B$.

Since $K$ is compact, it cannot contain an infinite subset of $B$. Without loss of generality, we may assume that $B\cap K=\emptyset$. Then $FP_{2}(X)\setminus B$ is an open set containing $K$ but not containing $U_{n}$ for any $n\in\mathbb{N}$, which is a contradiction. Therefore, we conclude that $X\setminus I(X)$ is bounded in $X$, as required.
\end{proof}

If the closure of every bounded set in a space $Z$ is compact,
then $Z$ is called a $\mu$-space.

\begin{corollary}
If $FP_{2}(X)$ is of pointwise countable type and $X$ is a $\mu$-space, then $X\setminus I(X)$ is compact.
\end{corollary}

\begin{question}
If $FP_{2}(X)$ is of pointwise countable type, is $X\setminus I(X)$  compact?
\end{question}

\begin{theorem}\label{t4}\cite{EN2012}
Let $X$ be a $T_{1}$-space and let $w=x_{1}^{\epsilon_{1}}x_{2}^{\epsilon_{2}}\cdots x_{n}^{\epsilon_{n}}$ be a reduced word in $FP_{n}(X)$, where $x_{i}\in X$ and $\epsilon_{i}=\pm 1$, for all $i=1, 2, \cdots, n$, and if $x_{i}=x_{i+1}$ for some $i=1, 2, \cdots, n-1$, then $\epsilon_{i}=\epsilon_{i+1}$. Then the collection $\mathscr{B}$ of all sets of the form $U_{1}^{\epsilon_{1}}U_{2}^{\epsilon_{2}}\cdots U_{n}^{\epsilon_{n}}$, where, for all $i=1, 2, \cdots, n$, the set $U_{i}$ is a neighborhood of $x_{i}$ in $X$ when $\epsilon_{i}=1$ and $U_{i}=\{x_{i}\}$ when $\epsilon_{i}=-1$ is a base for the neighborhood system at $w$ in $FP_{n}(X)$.
\end{theorem}

\begin{theorem}\label{t5}\cite{EN2012}
Let $X$ be a $T_{1}$-space and let $w=\epsilon_{1}x_{1}+\epsilon_{2}x_{2}+\cdots+\epsilon_{n}x_{n}$ be a reduced word in $AP_{n}(X)$, where $x_{i}\in X$ and $\epsilon_{i}=\pm 1$, for all $i=1, 2, \cdots, n$, and if $x_{i}=x_{j}$ for some $i, j=1, 2, \cdots, n$, then $\epsilon_{i}=\epsilon_{j}$. Then the collection $\mathscr{B}$ of all sets of the form $\epsilon_{1}U_{1}+\epsilon_{2}U_{2}+\cdots+\epsilon_{n}U_{n}$, where, for all $i=1, 2, \cdots, n$, the set $U_{i}$ is a neighborhood of $x_{i}$ in $X$ when $\epsilon_{i}=1$ and $U_{i}=\{x_{i}\}$ when $\epsilon_{i}=-1$ is a base for the neighborhood system at $w$ in $AP_{n}(X)$.
\end{theorem}

Let $X$ be a set. Define $j_{2}, k_{2}: X\times X\longrightarrow F_{a}(X)$ by $j_{2}(x, y)=x^{-1}y$ and $k_{2}(x, y)=yx^{-1}$.

By a {\it quasi-uniform space} $(X, \mathscr{U})$ we mean the natural analog of a {\it uniform space} obtained by dropping the symmetry axiom. We also recall that the {\it universal quasi-uniformity} $\mathscr{U}^{\ast}$ of a space $X$ is the finest quasi-uniformity on $X$ that induces on $X$ its original topology.

\begin{theorem}\label{t6}\cite{EN2013}
Let $X$ be a topological space. Then the collection $\mathscr{B}$ of sets $j_{2}(U)\cup k_{2}(U)$ for $U\in\mathscr{U}^{\ast}$ is a base of neighborhoods at the identity $e$ in $FP_{2}(X)$.
\end{theorem}

\begin{theorem}\label{t8}\cite{LFC2012}
Let $X$ be a topological space. Then the collection $\mathscr{B}$ of sets $W(U)$ for $U\in\mathscr{U}^{\ast}$ is a base of neighborhoods at the identity $e$ in $AP_{2}(X)$, where $W(U)=\{-x+y: (x, y)\in U\}$.
\end{theorem}

The proof of the following theorem is a modification of \cite[Lemma 2.10]{T2003}. However, for the completeness, we give out the proof.

\begin{theorem}\label{l4}
Let $X$ be an arbitrary $T_{1}$ space and $\varphi: FP(X)\rightarrow AP(X)$ be the canonical homomorphism $FP(X)$ onto $AP(X)$. Then the restriction $\varphi_{2}|_{FP_{2}(X)}$ is a perfect mapping of $FP_{2}(X)$ onto $AP_{2}(X)$.
\end{theorem}

\begin{proof}
Clearly, for each $y\in AP_{2}(X)$, we have $|\varphi_{2}^{-1}(y)|\leq 2$, hence the fibers of $\varphi_{2}$ are compact. Next we shall show that $\varphi_{2}$ is closed at every point $g\in AP_{2}(X)$.

{\bf Case 1:} $g\in AP_{1}(X)\setminus\{0\}$, where 0 denotes the neutral element of $AP(X)$.

Obviously, the set $Y=X\cup X^{-1}$ is clopen in $FP_{2}(X)$, $Z=X\cup (-X)$ is clopen in $AP_{2}(X)$, $\varphi_{2}^{-1}(Z)=Y$ and the restriction of $\varphi_{2}$ to $Y$ is a homeomorphism from $Y$ onto $Z$. Therefore, $\varphi_{2}$ is closed at every point of $Z$.

{\bf Case 2:} $g\in AP_{2}(X)\setminus AP_{1}(X)$.

Let $g=\varepsilon_{1}x_{1}+\varepsilon_{2}x_{2}\neq 0$, where $x_{1}, x_{2}\in X$ and $\varepsilon_{1}, \varepsilon_{2}=\pm 1$. Then we have $$\varphi_{2}^{-1}(g)=\{x_{1}^{\varepsilon_{1}}x_{2}^{\varepsilon_{2}}, x_{2}^{\varepsilon_{2}}x_{1}^{\varepsilon_{1}}\}.$$
Let $O$ be an open set in $FP_{2}(X)$ containing $\varphi_{2}^{-1}(g)$. By Theorem~\ref{t4}, there exist subsets $U_{1}$  and $U_{2}$ in $X$ satisfying the following conditions (1)-(3):

(1) For each $i=1, 2$, if $\varepsilon_{i}=1$ then $U_{i}$ is an open neighborhood of $x_{i}$ in $X$; if $\varepsilon_{i}=-1$ then $U_{i}=\{x_{i}\}$.

(2) $U_{1}^{\varepsilon_{1}}U_{2}^{\varepsilon_{2}}\cup U_{2}^{\varepsilon_{2}}U_{1}^{\varepsilon_{1}}\subset O$;

(3) If $\varepsilon_{1}+\varepsilon_{2}=0$, then $U_{1}\cap U_{2}=\emptyset$.

By Theorem~\ref{t5}, the set $V=\varepsilon_{1}U_{1}+\varepsilon_{2}U_{2}$ is an open neighborhood of $g$ in $AP_{2}(X)$, and one easily checks that $$\varphi_{2}^{-1}(V)=U_{1}^{\varepsilon_{1}}U_{2}^{\varepsilon_{2}}\cup U_{2}^{\varepsilon_{2}}U_{1}^{\varepsilon_{1}}\subset O.$$Therefore, $\varphi_{2}$ is closed at $g$.

{\bf Case 3:} $g=0$.

Clearly, we have $\varphi_{2}^{-1}(0)=\{e\}$. Let $O$ be a neighborhood of $e$ in $FP_{2}(X)$. By Theorem~\ref{t6}, there exists $U\in\mathscr{U}^{\ast}$ such that $j_{2}(U)\cup k_{2}(U)\subset O.$ It follows from Theorem~\ref{t8} that $W(U)=\{-x+y: (x, y)\in U\}$  is an open neighborhood of 0 in $AP_{2}(X)$. Obviously, we have $\varphi_{2}^{-1}(W)=j_{2}(U)\cup k_{2}(U)\subset O$. Therefore, $\varphi_{2}$ is closed at $0$.
\end{proof}

{\bf Remark \cite{T2003}}: It is worth mentioning that Theorem~\ref{l4} cannot be generalized to $n=3$. Indeed, let $X$ be an infinite $T_{1}$-space. Denote by $\varphi_{3}$ the restriction of the canonical homomorphism $\varphi: FP(X)\rightarrow AP(X)$ to $FP_{3}(X)$. Take an arbitrary point $x_{0}\in X$. Then $\varphi_{3}^{-1}(x_{0})$ contains the closed and discrete set $\{x x_{0}x^{-1}: x\in X\}$, hence $\varphi_{3}^{-1}(x_{0})$ is not compact. Therefore, $\varphi_{3}$ fails to be a perfect mapping.

It is well known that countable type and local compactness are both invariant and inverse invariant properties of perfect mappings, hence we have the following result.

\begin{proposition}\label{p7}
Let $X$ be an arbitrary Hausdorff space. Then we have
\begin{enumerate}
\item If $AP_{2}(X)$ is of pointwise countable type, then $FP_{2}(X)$ is of pointwise countable type;

\item $AP_{2}(X)$ is of countable type if and only if $FP_{2}(X)$ is of countable type;

\item $AP_{2}(X)$ is locally compact if and only if $FP_{2}(X)$ is locally compact.
\end{enumerate}
\end{proposition}

\begin{theorem}
Let $X$ be an arbitrary Hausdorff $\mu$-space. If $AP_{2}(X)$ or $FP_{2}(X)$ is locally compact, then $X$ is homeomorphic to the topological sum of a compact space and a discrete space.
\end{theorem}

\begin{proof}
By Theorem~\ref{t7} and Proposition~\ref{p7}, the set $X\setminus I(X)$ is bounded in $X$, then $X\setminus I(X)$ is compact since $X$ is a $\mu$-space. Since $AP_{2}(X)$ or $FP_{2}(X)$ is locally compact, $X$ is local compact. Then it is easy to see that $X$ is homeomorphism to the topological sum of a compact space and a discrete space.
\end{proof}

The following two questions are still unknown.

\begin{question}
If $X$ is homeomorphic to the topological sum of a compact space and a discrete space,  is $AP_{2}(X)$ or $FP_{2}(X)$ locally compact?
\end{question}

\begin{question}
Can we find a topological property $\mathcal{P}$ such that $AP_{2}(X)$ or $FP_{2}(X)$ is of point-countable type or countable type if and only if $X$ has the property $\mathcal{P}$?
\end{question}

\end{document}